\documentclass[11pt]{amsart}
\usepackage[T1]{fontenc}
\usepackage[utf8]{inputenc}
\usepackage{lmodern}
\usepackage{amsmath,amssymb,amsthm,mathtools}
\usepackage{tikz}
\usepackage[a4paper,margin=30mm]{geometry}
\usepackage[hidelinks]{hyperref}
\hypersetup{pdftitle={On multilinear forms for mod p representations of GL2 over a p-adic field},pdfauthor={Hao Liu, Haoran Wang, and Yuting Yang}}
\newtheorem{theorem}{Theorem}[section]
\newtheorem{corollary}[theorem]{Corollary}
\newtheorem{proposition}[theorem]{Proposition}
\newtheorem{lemma}[theorem]{Lemma}

\theoremstyle{definition}

\theoremstyle{remark}
\newtheorem{remark}[theorem]{Remark}
\DeclareMathOperator{\GL}{GL}
\DeclareMathOperator{\Hom}{Hom}
\DeclareMathOperator{\End}{End}
\DeclareMathOperator{\ind}{ind}
\DeclareMathOperator{\Ind}{Ind}
\DeclareMathOperator{\Sym}{Sym}

\newcommand{\FF}{\mathbb F}
\newcommand{\cO}{\mathcal O}
\numberwithin{equation}{section}
\allowdisplaybreaks[1]
\title[Multilinear forms for mod \(p\) representations]{On multilinear forms for mod \(p\) representations of \(\GL_2\) over a \(p\)-adic field}
\author{Hao Liu}
\address{Department of Mathematical Sciences, Tsinghua University, Beijing, 100084}
\email{liuhao21@mails.tsinghua.edu.cn}
\author{Haoran Wang}
\address{Academy for Multidisciplinary Studies, Capital Normal University, Beijing 100048}
\email{haoran@cnu.edu.cn}
\author{Yuting Yang}
\address{School of Mathematical Sciences, Capital Normal University, Beijing 100048}
\email{yangyuting0223@163.com}
\date{}
\begin{document}
\begin{abstract}
Let \(F\) be a finite extension of \(\mathbb Q_p\), for any prime \(p\), and let
\(L/F\) be a ramified quadratic extension. We compute the spaces of
\(\cO_L^\times\)-equivariant linear forms on the representation $\ind^G_{KZ}\sigma/(\Phi-\lambda)$. The only nonzero cases arise
from one-dimensional quotients. Consequently, the restriction of an
infinite-dimensional irreducible smooth mod \(p\) representation
of \(\GL_2(F)\) with a central character to \(\cO_L^\times\) has no character quotient. We apply this vanishing result to show that a finite tensor product of smooth mod \(p\) representations of \(\GL_2(F)\) admits no nonzero invariant linear form whenever at least one factor is infinite-dimensional and irreducible with a central
character.
\end{abstract}
\maketitle

\section{Introduction}

Let $F$ be a finite extension of $\mathbb{Q}_p$, and let $G$ be a  $p$-adic Lie group over $F$. The study of $G$-invariant multilinear forms on smooth complex representations of $G$ is a classical problem in representation theory and has profound applications in number theory.
For three irreducible admissible smooth complex representations of $G$, where $G$ is either $\mathrm{GL}_2(F)$ or $D^{\times}$ (here $D$ is the nonsplit quaternion algebra over $F$), Prasad proves the multiplicity-one theorem \cite[Theorem~1.1]{Prasad}
\[
\dim_{\mathbb{C}}\operatorname{Hom}_G(\pi_1\otimes_{\mathbb{C}}\pi_2\otimes_{\mathbb{C}}\pi_3,\mathbb{C})\leq 1,
\]
together with a dichotomy between the $\mathrm{GL}_2(F)$ and $D^{\times}$ cases (\cite[Theorem~1.2]{Prasad}).

In characteristic $p$, the analogue of Prasad's multiplicity-one result is proved by Allen--Wake--Zhang \cite[Corollary~4.1]{AllenWakeZhang2026} in the case $G = D^{\times}$ and $p>2$. In this paper, we consider the case $G = \mathrm{GL}_2(F)$. For $G = \mathrm{GL}_2(\mathbb{Q}_p)$, it follows from \cite[Corollary~1.4]{Fan} that, for any $n\geq 1$ and smooth
\(\overline{\mathbb{F}}_p\)-representations \(\pi_1,\ldots,\pi_n\) of $G$, 
\[
 \Hom_{\GL_2(\mathbb Q_p)}
 (\pi_1\otimes\cdots\otimes\pi_n,\overline{\mathbb{F}}_p)=0,
\]
whenever at least one $\pi_i$ is an infinite-dimensional irreducible admissible representation. Over a general finite extension \(F/\mathbb Q_p\), the argument of \cite[Lemma~3.5 and Proposition~4.1]{Fan} also gives this vanishing whenever at least one factor is an irreducible admissible principal series or special series representation. These results leave open the case in which all infinite-dimensional factors are supersingular over general $F$. Fan's treatment of supersingular factors over $ \mathbb{Q}_p$ relies on the structural results of Pa\v{s}k\=unas \cite{Paskunas}, and this approach does not directly extend to a general finite extension $F$.

Our first main result is the following theorem. Throughout this paper, we write $\mathbb{F} = \overline{\mathbb{F}}_p$.
\begin{theorem} \label{thm:main}
Let $n\geq 1$, and let $\pi_1,\ldots,\pi_n$ be smooth $\mathbb{F}$-representations of $G=\mathrm{GL}_2(F)$. Suppose that at least one $\pi_i$ is infinite-dimensional and irreducible, with a central character. Then
\[
\operatorname{Hom}_G(\pi_1\otimes_{\mathbb{F}}\cdots\otimes_{\mathbb{F}}\pi_n,\mathbb{F})=0.
\]
\end{theorem}

Our proof avoids the obstacle by combining an explicit restriction calculation for the unit group $\mathcal{O}_L^\times$ of a ramified quadratic extension $L/F$ with a tensor-product property of pro-$p$ coinvariants. The restriction calculation supplies an open pro-$p$ subgroup $P$ on which an infinite-dimensional factor has zero coinvariants. The tensor-product property then transfers this vanishing to the full tensor product. 

\medskip
\noindent
\textbf{The restriction calculation.}
Let $L/F$ be a ramified quadratic extension, embedded in $M_2(F)$ as in  \S\ref{sec:restriction}, so that $\mathcal{O}_L^\times$ is a subgroup of $K=\mathrm{GL}_2(\mathcal{O}_F)$. Let $Z$ be the center of $G$, and let $k_F$ be the residue field of $F$. A Serre weight $\sigma$ is an irreducible $\mathbb{F}$-representation of $\mathrm{GL}_2(k_F)$, inflated to $K$. Fix a uniformizer $\varpi_F$ and extend $\sigma$ to $KZ$ by requiring $\varpi_F I_2$ to act trivially. By \cite[Proposition~8]{BL94},
\[
\operatorname{End}_G(\operatorname{ind}_{KZ}^G\sigma)\simeq \mathbb{F}[\Phi],
\]
where $\operatorname{ind}$ denotes compact induction and $\Phi$ is the Hecke operator normalized as in \S\ref{subsec:hecke}. The key calculation of this paper is the following theorem.

\begin{theorem}\label{thm:compute}
Let \(\sigma\) be a Serre weight, let \(\lambda\in\FF\), and let
\(\eta:\cO_L^\times\to\FF^\times\) be a smooth character. Then
\[
 \dim_{\FF}\Hom_{\cO_L^\times}
 \left(\left(\ind^G_{KZ}\sigma/(\Phi-\lambda)\right)|_{\cO_L^\times},\eta\right)
 =
 \begin{cases}
 1,& \text{ if } \dim\sigma=1,\ \lambda^2=1,\ \eta=\sigma|_{\cO_L^\times},\\
 0,&\text{otherwise}.
 \end{cases}
\]
\end{theorem}

The proof combines the Mackey decomposition for the restriction to $\cO_L^\times Z$ of \(\ind^G_{KZ}\sigma\) with an explicit description of the Hecke operator $\Phi$. Equivariant linear forms on \(\ind^G_{KZ}\sigma\) are described
by sequences indexed by the integers, and precomposition by $\Phi$ is given by explicit relations (Lemma \ref{lem:dual-hecke}). These relations force vanishing except in the one-dimensional cases listed in Theorem~\ref{thm:compute}. Our strategy is inspired by \cite{Chen-Wang}.

Theorem \ref{thm:compute} leads to the following consequence.

\begin{corollary} \label{cor::intro1}
Let \(\pi\) be an infinite-dimensional irreducible smooth
representation of \(G\) with a central character, and let
\(\eta:\cO_L^\times\to\FF^\times\) be a smooth character. Then
\[
 \Hom_{\cO_L^\times}(\pi|_{\cO_L^\times},\eta)=0.
\]
\end{corollary}

This corollary also yields vanishing for characters of \(L^\times\) and for invariant forms on any open pro-\(p\)
subgroup of \(\cO_L^\times\)
(Corollaries~\ref{cor:torus-vanishing} and~\ref{cor:open-subgroup}).
For \(F=\mathbb Q_p\) and $p>2$, Morra proves related vanishing results
for restrictions to \(\cO_L^\times\)
\cite[Theorem~1.5]{Morra}, allowing any quadratic extension $L  /\mathbb{Q}_p$. Our calculation treats the ramified case over arbitrary $F$ and for every prime $p$, including $p=2$.

\medskip
\noindent
\textbf{Multilinear forms.}
We next use the restriction calculation to prove Theorem \ref{thm:main}. The proof rests on a tensor-product property of pro-$p$ coinvariants. Let $P$ be a pro-$p$ group, and let $V_1,V_2$ be smooth representations of $P$. We prove (see Lemma~\ref{lem:tensor-vanishing}) that
\[
\FF\otimes_{\FF[\![P]\!]}V_1=0
 \quad\Longrightarrow\quad
\FF\otimes_{\FF[\![P]\!]}(V_1\otimes_{\mathbb F}V_2)=0.
\]
This extends \cite[Lemma~3.1]{Fan} from $\mathbb{Z}_p$ to arbitrary pro-$p$ groups.

To prove Theorem \ref{thm:main}, we choose an open pro-$p$ subgroup \(P \subseteq \cO_L^\times\) and reorder the factors so that \(\pi_1\) is infinite-dimensional and irreducible with a central character. Corollary~\ref{cor:open-subgroup} gives $ \Hom_{P}(\pi_1|_{P},\mathbb{F})=0$ and hence $\FF\otimes_{\FF[\![P]\!]}(\pi_1|_P)=0$.
Applying Lemma~\ref{lem:tensor-vanishing} to \(V_1=\pi_1|_{P}\) and
\(V_2=(\pi_2\otimes\cdots\otimes\pi_n)|_{P}\), with
\(V_2=\FF\) when \(n=1\), we deduce that the tensor product has no nonzero \(P\)-invariant
linear form, and hence no nonzero \(G\)-invariant linear form.

Gabriel Dospinescu also pointed out that Theorem~\ref{thm:main} can be deduced directly from our restriction results by smoothness; see Remark~\ref{rem:dospinescu}.

\bigskip

The paper is organized as follows. In \S\ref{sec:restriction}, we
compute the spaces of equivariant linear forms for
restrictions to \(\mathcal{O}_L^\times\) and derive their
vanishing consequences. In  \S\ref{sec:multilinear}, we establish the
tensor-product property for pro-\(p\) coinvariants and
apply it to invariant multilinear forms.

\subsection*{Acknowledgements}

We thank Yikun Fan and Vytautas Pa{\v{s}}k{\={u}}nas for sending us the preprint \cite{Fan} and for helpful discussions. We also thank Gabriel Dospinescu for his helpful comments on an earlier version of this paper and for pointing out how to deduce Theorem~\ref{thm:main} directly from our restriction results; see Remark~\ref{rem:dospinescu}. This work is supported by National Key R\&D Program of China 2023YFA1009702, National Natural Science Foundation of China (Grant No. 12371011), and the Beijing Natural Science Foundation (Grant No. 1262003).
The authors use ChatGPT 6 to help with language editing and proofreading.

\section{Restriction to the unit group of a quadratic extension}
\label{sec:restriction}

Let \(L/F\) be a ramified quadratic extension. Write \(\cO_F\) and \(\cO_L\) for the rings of integers
of $F$ and $L$, respectively. Fix a uniformizer \(\varpi_F\) of \(F\). Let \(v_F\) be the valuation on $F$ such that
\(v_F(\varpi_F)=1\). Let $\varpi_L$ be a uniformizer of $L$. The minimal polynomial of \(\varpi_L\)
over \(F\) is Eisenstein, so $\cO_L=\cO_F[\varpi_L]$ and 
\begin{equation}\label{eq:uniformizer}
 \varpi_L^2=\tau\varpi_L+u\varpi_F
\end{equation}
for $\tau\in\varpi_F\cO_F$ and $u\in\cO_F^\times$. We use the embedding
\begin{equation}\label{eq:iota}
\begin{aligned}
 \iota\colon L&\longrightarrow M_2(F),\\
 a+b\varpi_L &\longmapsto
 \begin{pmatrix}a&b\\u\varpi_F b&a+\tau b\end{pmatrix}.
\end{aligned}
\end{equation}
We identify \(L^\times\) with its image. Then  \(\cO_L^\times\subset K:=\GL_2(\cO_F)\). Let $Z$ be the center of $G$. Set $H=\cO_L^\times Z$. Then \([L^\times:H]=2\).
Let \(k_F=\cO_F/\varpi_F\cO_F\). Define $q=|k_F|$ and \(f=\log_pq\). Fix an
embedding \(k_F\hookrightarrow\FF\).
For \(a\in\cO_F\), write \(\bar a\) for its residue in $k_F$ (also in $\FF$), and for
\(t\in k_F\), we write \([t]\) for its Teichm\"uller lift, with \([0]=0\).

\subsection{The Serre weight and the Hecke operator}
\label{subsec:hecke}
A Serre weight is an irreducible \(\FF\)-representation of
\(\GL_2(k_F)\), inflated to \(K\). We extend it to \(KZ\) by requiring
\(\varpi_F I_2\) to act trivially.

For \(\vec r=(r_1,\ldots,r_f)\in\{0,\ldots,p-1\}^f\), let
\(\sigma_{\vec r}\) be the representation with underlying space
\[
 V_{\sigma_{\vec r}}=
 \bigotimes_{i=1}^f\Sym^{r_i}(\FF^2).
\]
We realize the \(i\)-th factor as the space of homogeneous polynomials of
degree \(r_i\) in \(x_i,y_i\), with action
\[
 \begin{pmatrix}a&b\\c&d\end{pmatrix}
 \cdot\bigotimes_{i=1}^f f_i(x_i,y_i)
 =\bigotimes_{i=1}^f
 f_i\left(\bar a^{p^{i-1}}x_i+\bar c^{p^{i-1}}y_i,
          \bar b^{p^{i-1}}x_i+\bar d^{p^{i-1}}y_i\right).
\]
Every Serre weight has the form
\(\sigma_{\vec r}\otimes\det^m\), where \(0\le m\le p^f-2\);
see \cite[Proposition 1]{BL94}.

For the following computations, assume \(\sigma=\sigma_{\vec r}\).

For $\vec{j}=(j_1,\ldots,j_f)$, we write $0\le\vec{j}\le\vec{r}$ if $0\le j_i\le r_i$ for all $1\le i \le f$. For $0\le\vec{j}\le\vec{r}$, let 
\(
 \vec e_{\vec j}=\bigotimes_{i=1}^f x_i^{r_i-j_i}y_i^{j_i}
\).
Then the vectors \(\vec e_{\vec j}\) form a basis of \(V_\sigma\).
Define $ R(\vec j)=j_1+j_2p+\cdots+j_fp^{f-1}$. Let \(\zeta\) be the central character of $\sigma$. Then for $a\in\cO_F^{\times}$, 
\(\zeta(a)=\bar a^{R(\vec r)}
\).
Define the projections
\[
 \mathbf X\left(\sum_{\vec j}c_{\vec j}\vec e_{\vec j}\right)
   =c_{\vec0}\vec e_{\vec0},\qquad
 \mathbf Y\left(\sum_{\vec j}c_{\vec j}\vec e_{\vec j}\right)
   =c_{\vec r}\vec e_{\vec r}.
\]

Let \(\ind^G_{KZ}\sigma\) be the space of locally constant functions
\(f:G\to V_\sigma\) with support compact modulo \(KZ\), satisfying
\[f(kg)=\sigma(k)f(g),\; \forall k\in KZ,\; \forall g \in G.\] 
The group $G$ acts on \(\ind^G_{KZ}\sigma\) by right translations: \((gf)(g')=f(g'g)\).
For \(g\in G\) and \(v\in V_\sigma\), let \([g,v]\) be the function
supported on \(KZg^{-1}\) with \([g,v](g^{-1})=v\).
Then \([gk,v]=[g,\sigma(k)v]\) for all \(k\in KZ\).
\begin{proposition}[Barthel--Livn\'e]\label{prop:hecke}
There exists a unique \(\Phi\in\End_G(\ind^G_{KZ}\sigma)\) such that 
\begin{equation}\label{eq::hecke}
 \Phi([1,v])
 =\left[\begin{pmatrix}\varpi_F&0\\0&1\end{pmatrix}^{-1},
          \mathbf Y(v)\right]+\sum_{t\in k_F}
 \left[\begin{pmatrix}1&[t]\\0&\varpi_F\end{pmatrix}^{-1},
          \mathbf X\!\left(\begin{pmatrix}1&[t]\\0&1\end{pmatrix}v\right)\right]
\end{equation}
for all $v\in V_{\sigma}$. Moreover, there is an isomorphism of algebras
\[
\End_G(\ind^G_{KZ}\sigma)\cong\FF[\Phi].
\]
\end{proposition}
\begin{proof}
The operator $\Phi$ is the operator $T$ defined in
\cite[\S 3]{BL94}. See also \cite[Proposition 8 and Equation (19)]{BL94}. Since $\ind^G_{KZ}\sigma$ is generated by $[1,V_{\sigma}]$ as a $G$-representation, $\Phi$ is uniquely determined by \eqref{eq::hecke}. 
\end{proof}

\subsection{Mackey decomposition}\label{subsec:mackey}

For \(k\in\mathbb Z\), define
\(
 s_k=\bigl(\begin{smallmatrix}\varpi_F^k&0\\0&1\end{smallmatrix}\bigr).
\)
\begin{lemma}\label{lem:cosets}
The matrices \(s_k\), for \(k\in\mathbb Z\), form a set of
representatives for \(KZ\backslash G/H\).
\end{lemma}
\begin{proof}
For $g_1,g_2\in G$, we write $g_1\sim g_2$ if $KZg_1 H=KZg_2 H$. Let \(B\) be the upper triangular subgroup of \(G\). Given $g\in G$, by the Iwasawa
decomposition \(G=KB\), there exists
\(\bigl(\begin{smallmatrix}a&c\\0&d\end{smallmatrix}\bigr)\in B\) such that \(\bigl(\begin{smallmatrix}a&c\\0&d\end{smallmatrix}\bigr)\sim g\).
Write \(a/d=\varepsilon\varpi_F^k\) with
\(\varepsilon\in\cO_F^\times\) and \(k\in\mathbb Z\).
Since 
\[
\begin{pmatrix}d^{-1}&0\\0&d^{-1}\end{pmatrix}\begin{pmatrix}\varepsilon^{-1}&0\\0&1\end{pmatrix}\begin{pmatrix}a&c\\0&d\end{pmatrix}=\begin{pmatrix}\varpi_F^k&n\\0&1\end{pmatrix}
\]
for $n=\varepsilon^{-1}c/d\in F$, we have 
\[
g\sim \beta_{k,n}:=\begin{pmatrix}\varpi_F^k&n\\0&1\end{pmatrix}.
\]
If $n\in \mathcal O_F$, then 
\[
 \begin{pmatrix}1&-n\\0&1\end{pmatrix}\beta_{k,n}
 =s_k,
\]
which implies $g\sim s_k$. 

It remains to deal with the  case when  $n\not\in\mathcal O_F$.
If \(v_F(n)\ge k\), set
\(h=1-\varpi_F^{-k}n(\varpi_L-\tau)\in\cO_L^\times\).
Then by \eqref{eq:uniformizer},
\begin{equation}\label{eq:coset-first}
 \beta_{k,n}\iota(h)s_k^{-1}
 =\begin{pmatrix}
 1+\tau n\varpi_F^{-k}-un^2\varpi_F^{1-2k}&0\\
 -un\varpi_F^{1-2k}&1
 \end{pmatrix}\in K.
\end{equation}
This implies  $g\sim s_k$. If \(v_F(n)\le k-1\), put
\[
 k'=k-2v_F(n),\qquad
 h=n^{-1}-u^{-1}\varpi_F^{k-1}n^{-2}\varpi_L.
\]
Since $nh$ is a unit of \(\cO_L\), the element \(h\) belongs to \(H\).
Using \eqref{eq:uniformizer} we have
\[
 \beta_{k,n}\iota(h)s_{k'}^{-1}
 =\begin{pmatrix}
 0&1-u^{-1}\varpi_F^{2k-1}n^{-2}
       -\tau u^{-1}\varpi_F^{k-1}n^{-1}\\
 -\varpi_F^{k-k'}n^{-2}&n^{-1}
       -\tau u^{-1}\varpi_F^{k-1}n^{-2}
 \end{pmatrix}\in K.
\]
This proves $g\sim s_{k'}$. Therefore, every double coset in \(KZ\backslash G/H\) has a representative of the form $s_k$ with $k\in\mathbb Z$.

For uniqueness, suppose that \(KZs_kH=KZs_lH\).
Since \(Z\subset H\), there exists $h\in H$ such that $ s_k\iota(h)s_l^{-1}\in K$. Since $\cO_L=\cO_F[\varpi_L]$, there exist \(b\in\cO_F^\times\), \(c\in\cO_F\) and $m\in \mathbb Z$ such that
\(h=\varpi_F^m(b+c\varpi_L)\). Then
\[
 s_k\iota(h)s_l^{-1}
 =\begin{pmatrix}
  \varpi_F^{k-l+m}b&\varpi_F^{k+m}c\\
  u\varpi_F^{1-l+m}c&\varpi_F^m(b+\tau c)
 \end{pmatrix}.
\]
We have $v_F(\det (s_k\iota(h)s_l^{-1}))=2m+k-l$, so \(2m+k-l=0\).
Since both \(b\) and \(b+\tau c\) are units in $\mathcal O_F$,
integrality of the diagonal entries gives \(m\ge0\) and
\(k-l+m\ge0\). Their sum is zero, hence \(m=0\) and \(k=l\).
\end{proof}

For $k\in\mathbb Z$, define $H_k=H\cap s_k^{-1}KZs_k$. Let $\sigma_k$ be the representation of $H_k$ with the same underlying space as $\sigma$, defined by $ \sigma_k(h)=\sigma(s_khs_k^{-1})$ for $h\in H_k$.
Since \(KZ\) is open in $G$, the Mackey decomposition
 gives
\begin{equation}\label{eq:mackey}
 \ind^G_{KZ}\sigma|_H\simeq\bigoplus_{k\in\mathbb Z}\ind_{H_k}^H\sigma_k,
\end{equation}
where the inclusion $\ind_{H_k}^H\sigma_k\hookrightarrow  \ind^G_{KZ}\sigma$ is determined by
\(v\mapsto[s_k^{-1},v]\).

\begin{lemma}\label{lem:subgroups}
We have 
\begin{equation}\label{eq:H_k}
 H_k=
 \begin{cases}
 (1+\varpi_L\varpi_F^{k-1}\cO_F)Z,&k\ge1,\\
 (1+\varpi_L\varpi_F^{-k}\cO_F)Z,&k\le0.
 \end{cases}
\end{equation}
In particular \(H_0=H_1=H\). For \(a\in\cO_F^\times\),
\(b\in\cO_F\), and \(n\in\mathbb Z\), we have
\begin{align*}
 \sigma_k\bigl((1+\varpi_L\varpi_F^{k-1}b)\varpi_F^n a\bigr)
 \vec e_{\vec j}
 &=\bar a^{R(\vec r)}
   \bigotimes_{i=1}^f
   \left(x_i+(\bar u\bar b)^{p^{i-1}}y_i\right)^{r_i-j_i}y_i^{j_i}
   &&(k\ge1),\\
 \sigma_k\bigl((1+\varpi_L\varpi_F^{-k}b)\varpi_F^n a\bigr)
 \vec e_{\vec j}
 &=\bar a^{R(\vec r)}
   \bigotimes_{i=1}^f
   x_i^{r_i-j_i}\left(\bar b^{p^{i-1}}x_i+y_i\right)^{j_i}
   &&(k\le0).
\end{align*}
\end{lemma}
\begin{proof}
An element of \(\cO_L^\times\) can be written as \(a+b\varpi_L\)
with \(a\in\cO_F^\times\) and \(b\in\cO_F\).
Conjugating by \(s_k\) shows that \((a+b\varpi_L) Z\subset H_k\) precisely when
\(v_F(b)\ge\max\{0,-k,k-1\}\). 
Factoring out \(a\) proves the descriptions of \(H_k\).
The description of the action follows from \eqref{eq:uniformizer} and \eqref{eq:iota}.
\end{proof}

\begin{samepage}
Let \(\Phi_k\) be the restriction of \(\Phi\) to the \(k\)-th
summand of \eqref{eq:mackey}.
For \(k,l\in\mathbb Z\), let \(\Phi_k^l\) be this restriction followed
by projection onto the \(l\)-th summand. Under the identification
\eqref{eq:mackey}, $\Phi_k^l$ is described by the following commutative diagram:
\[
\begin{tikzpicture}[>=stealth,baseline=(current bounding box.center)]
 \node (source) at (0,0) {\(\displaystyle\ind_{H_k}^H\sigma_k\)};
 \node (sumsource) at (6.6,0)
  {\(\displaystyle\bigoplus_{j\in\mathbb Z}\ind_{H_j}^H\sigma_j\)};
 \node (target) at (0,-1.8) {\(\displaystyle\ind_{H_l}^H\sigma_l\)};
 \node (sumtarget) at (6.6,-1.8)
  {\(\displaystyle\bigoplus_{j\in\mathbb Z}\ind_{H_j}^H\sigma_j\)};
 \draw[->] (source) -- node[above] {inclusion} (sumsource);
 \draw[->] (sumsource) -- node[right] {\(\Phi\)} (sumtarget);
 \draw[->] (sumtarget) -- node[below] {projection} (target);
 \draw[->] (source) -- node[left] {\(\Phi_k^l\)} (target);
\end{tikzpicture}
\]
\end{samepage}
By abuse of notation, in the representation $\ind_{H_k}^H\sigma_k$, for $g\in H$ and $v\in\sigma_k$, we also write \([g,v]\) for the function
supported on \(H_kg^{-1}\) with \([g,v](g^{-1})=v\).
\begin{lemma}\label{lem:local-hecke}
For $k\in \mathbb Z$, we have \(\Phi_k=\Phi_k^{k+1}+\Phi_k^{k-1}\),
where the target summands are identified with their images in \eqref{eq:mackey}.
\hfill\begin{enumerate}
\item If \(k\ge1\),
\begin{align*}
 \Phi_k^{k+1}([1,\vec e_{\vec j}])
 &=\sum_{t\in k_F}
 [1-\varpi_L\varpi_F^{k-1}[t],
       (\bar u t)^{R(\vec r)-R(\vec j)}\vec e_{\vec r}],\\
 \Phi_k^{k-1}([1,\vec e_{\vec j}])
 &=[1,\delta_{\vec j,\vec0}\vec e_{\vec0}].
\end{align*}
\item If \(k\le0\),
\begin{align*}
 \Phi_k^{k+1}([1,\vec e_{\vec j}])
 &=[1,\delta_{\vec j,\vec r}\vec e_{\vec r}],\\
 \Phi_k^{k-1}([1,\vec e_{\vec j}])
 &=\sum_{t\in k_F}
 [1-\varpi_L\varpi_F^{-k}[t],t^{R(\vec j)}\vec e_{\vec0}].
\end{align*}
\end{enumerate}
Here \(\delta\) denotes the Kronecker symbol, and we use the convention $0^0=1$.

\end{lemma}
\begin{proof}
For each $t\in k_F$ and each $0\le\vec{j}\le\vec{r}$, we have
\[
 \mathbf X\!\left(\begin{pmatrix}1&[t]\\0&1\end{pmatrix}\vec e_{\vec j}\right)=t^{R(\vec j)}\vec e_{\vec0},
 \qquad
 \mathbf Y(\vec e_{\vec j})=\delta_{\vec j,\vec r}\vec e_{\vec r}.
\]
By \(G\)-equivariance and \eqref{eq::hecke},
\[
\begin{aligned}
 \Phi([s_k^{-1},\vec e_{\vec j}])&=s_k^{-1}\Phi([1,\vec e_{\vec j}])\\
 &=\sum_{t\in k_F}
 \left[\begin{pmatrix}\varpi_F^{k-1}&\varpi_F^{-1}[t]\\0&1\end{pmatrix}^{-1},
      \mathbf X\!\left(\begin{pmatrix}1&[t]\\0&1\end{pmatrix}\vec e_{\vec j}\right) \right]+[s_{k+1}^{-1},  \mathbf Y(\vec e_{\vec j})]\\
 &=\sum_{t\in k_F}
 \left[\begin{pmatrix}\varpi_F^{k-1}&\varpi_F^{-1}[t]\\0&1\end{pmatrix}^{-1},
       t^{R(\vec j)}\vec e_{\vec0}\right]+[s_{k+1}^{-1},\delta_{\vec j,\vec r}\vec e_{\vec r}].
\end{aligned}
\]
If \(k\le0\), a direct matrix calculation gives
\begin{align*}
\begin{pmatrix}\varpi_F^{k-1}&\varpi_F^{-1}[t]\\0&1\end{pmatrix}^{-1}
=
 \iota(1-\varpi_L\varpi_F^{-k}[t])s_{k-1}^{-1}
 \begin{pmatrix}
 1-u\varpi_F^{1-2k}[t]^2&-\tau\varpi_F^{-k-1}[t]^2\\
 -u\varpi_F^{2-2k}[t]&1-\tau\varpi_F^{-k}[t]
 \end{pmatrix}^{-1},
\end{align*}
where the last factor belongs to \(K\) and reduces to an upper unipotent
matrix modulo \(\varpi_F\), because \(\tau\in\varpi_F\cO_F\).
Since upper unipotent
matrices fix \(\vec e_{\vec0}\), we have, for \(k\le0\), 
\begin{equation}
 \Phi([s_k^{-1},\vec e_{\vec j}])
 =\sum_{t\in k_F}
 [\iota(1-\varpi_L\varpi_F^{-k}[t])s_{k-1}^{-1},
       t^{R(\vec j)}\vec e_{\vec0}]+[s_{k+1}^{-1},\delta_{\vec j,\vec r}\vec e_{\vec r}]\label{eq:negative-hecke}.
\end{equation}

Suppose now \(k>0\). 
For \(t\in k_F^\times\), we have
\[
 \begin{pmatrix}\varpi_F^{k-1}&\varpi_F^{-1}[t]\\0&1\end{pmatrix}^{-1}
 =\iota(1-\varpi_L\varpi_F^{k-1}[a])s_{k+1}^{-1}
 (\varpi_F I_2)C_{k,t}^{-1},
\]
where \(a=(\bar u t)^{-1}\) and
\[
 C_{k,t}=\begin{pmatrix}
  \varpi_F^{-1}(1-u[t][a])
   &[t]-\tau\varpi_F^{k-1}[t][a]-\varpi_F^{2k-1}[a]\\
  -u[a]&\varpi_F-\tau\varpi_F^k[a]
 \end{pmatrix}.
\]
Since \(u[t][a]\equiv1\pmod{\varpi_F}\), the entries of \(C_{k,t}\) are integral. We have
\[
 \overline{C}_{k,t}=\begin{pmatrix}*&t\\-\bar u a&0\end{pmatrix},
\]
whose determinant is \(\bar u a t=1\). Hence \(C_{k,t}\in K\),
and the first column of \(\overline{C_{k,t}^{-1}}\) is
\(\binom{0}{\bar u a}\). Therefore
\(\sigma(C_{k,t}^{-1})\vec e_{\vec0}
=(\bar u a)^{R(\vec r)}\vec e_{\vec r}\). Since \(\varpi_F I_2\)
acts trivially and \(t=(\bar u a)^{-1}\), we obtain
\[
\begin{aligned}
 \left[\begin{pmatrix}\varpi_F^{k-1}&\varpi_F^{-1}[t]\\0&1\end{pmatrix}^{-1},
 t^{R(\vec j)}\vec e_{\vec0}\right]=\left[\iota(1-\varpi_L\varpi_F^{k-1}[a])s_{k+1}^{-1},
 (\bar u a)^{R(\vec r)-R(\vec j)}\vec e_{\vec r}\right].
\end{aligned}
\]
When $t=0$,
\[
\begin{aligned}
 \left[\begin{pmatrix}\varpi_F^{k-1}&\varpi_F^{-1}[t]\\0&1\end{pmatrix}^{-1},
 t^{R(\vec j)}\vec e_{\vec0}\right]=[s_{k-1}^{-1},\delta_{\vec j,\vec0}\vec e_{\vec0}].
\end{aligned}
\]
Therefore
\[
\begin{aligned}
  \Phi([s_k^{-1},\vec e_{\vec j}])
 &=[s_{k-1}^{-1},\delta_{\vec j,\vec0}\vec e_{\vec0}]+\sum_{t\in k_F^{\times}}
 [\iota(1-\varpi_L\varpi_F^{k-1}[t])s_{k+1}^{-1},
       (\bar u t)^{R(\vec r)-R(\vec j)}\vec e_{\vec r}]+[s_{k+1}^{-1},\delta_{\vec j,\vec r}\vec e_{\vec r}]\\
 &=[s_{k-1}^{-1},\delta_{\vec j,\vec0}\vec e_{\vec0}]+\sum_{t\in k_F}
 [\iota(1-\varpi_L\varpi_F^{k-1}[t])s_{k+1}^{-1},
       (\bar u t)^{R(\vec r)-R(\vec j)}\vec e_{\vec r}].
\end{aligned}
\]
This proves the remaining formulas.
\end{proof}

\subsection{Equivariant linear forms}
\label{subsec:dual}

Continue to assume \(\sigma=\sigma_{\vec r}\), and let \(\zeta\)
be the central character of $\sigma$ (we view $\sigma$ as a representation of $KZ$). Let
\(\eta:\cO_L^\times\to\FF^\times\) be a smooth character. We extend $\eta$ to $\widetilde\eta\colon H\to\FF^{\times}$ by defining
\[
\widetilde\eta(u\varpi_F^m)=\eta(u)\quad(u\in\cO_L^\times,\ m\in \mathbb Z).
\]
If \(\eta|_{\cO_F^\times}\ne
\zeta|_{\cO_F^\times}\), then 
\[\Hom_{\cO_L^\times}(\ind^G_{KZ}\sigma, \eta)\subset\Hom_{\cO_F^\times}(\ind^G_{KZ}\sigma, \eta)=0.\]
By Frobenius reciprocity, we have
\[
\Hom_H(\ind_{H_k}^H\sigma_k,\widetilde\eta) \simeq\Hom_{H_k}(\sigma_k,\widetilde\eta).
\]
Define $ \ell_k\in\Hom_{\FF}(\sigma_k,\widetilde\eta)$ by
\[
 \ell_k(\vec e_{\vec j})=
 \begin{cases}
 \delta_{\vec j,\vec0},&k\ge1,\\
 \delta_{\vec j,\vec r},&k\le0,
 \end{cases}
\]
for $ 0\le\vec j\le \vec r$. 
\begin{lemma}\label{lem:local-hom}
Let
\(\eta:\cO_L^\times\to\FF^\times\) be a smooth character such that \(\eta|_{\cO_F^\times}=
\zeta|_{\cO_F^\times}\). We have $\Hom_{H_k}(\sigma_k,\widetilde\eta)=\FF\ell_k$ for each $k\in\mathbb Z$.
\end{lemma}
\begin{proof}
Suppose first $k\le0$. Let $N$ be the upper unipotent group of $\GL_2(k_F)$. Recall that $\sigma_k$ and $\sigma$ have the same underlying space. For $f\in\Hom_{\FF}(\sigma_k,\widetilde\eta)$, we claim that $f\in\Hom_{H_k}(\sigma_k,\widetilde\eta)$ if and only if $f$ factors through the $N$-coinvariants of $\sigma$. Indeed, since \(\eta|_{\cO_F^\times}=
\zeta|_{\cO_F^\times}\), $f$ is \(Z\)-equivariant. Since $\FF^{\times}$ has no non-trivial elements of $p$-power order, \(\widetilde\eta\) is trivial on
\(1+\varpi_L\cO_L\). By \eqref{eq:H_k}, $f$ is $H_k$-equivariant if and only if $f(\sigma_k(1+\varpi_L\varpi_F^{-k}b)x)=f(x)$ for all $b\in\cO_F$ and $x\in V_\sigma$. By Lemma~\ref{lem:subgroups}, this is also equivalent to requiring that $f$ factors through the $N$-coinvariants of $\sigma$.
By \cite[Lemma 2]{BL94} and its proof, we have
\[
(\sigma)_N=\sigma/\langle\vec e_{\vec j}: \vec j\neq\vec r \rangle.
\]
 So $\Hom_{H_k}(\sigma_k,\widetilde\eta)$ is one-dimensional and is generated by $\ell_k$.

When $k\ge1$, let $\overline{N}$ be the lower unipotent group of $\GL_2(k_F)$. Similarly, for $f\in\Hom_{\FF}(\sigma_k,\widetilde\eta)$, we have  $f\in\Hom_{H_k}(\sigma_k,\widetilde\eta)$ if and only if $f$ factors through the $\overline{N}$-coinvariants of $\sigma$. By \cite[Lemma 2]{BL94} and its proof we also have
\[
(\sigma)_{\overline N}=\sigma/\langle\vec e_{\vec j}: \vec j\neq\vec 0 \rangle.
\]
 So $\Hom_{H_k}(\sigma_k,\widetilde\eta)$ is one-dimensional and is generated by $\ell_k$.
\end{proof}
When \(\eta|_{\cO_F^\times}=
\zeta|_{\cO_F^\times}\), we also extend $\ell_k$ to a functional in $\Hom_H(\ind_{H_k}^H\sigma_k,\widetilde\eta)$ using \(\ell_k([h,v])=\widetilde\eta(h)\ell_k(v)\).
\begin{lemma}\label{lem:dual-hecke}
Suppose  \(\eta|_{\cO_F^\times}=
\zeta|_{\cO_F^\times}\). Under the identification
\[
 \Hom_H(\ind^G_{KZ}\sigma,\widetilde\eta)
 \simeq\prod_{k\in\mathbb Z}\Hom_H(\ind_{H_k}^H\sigma_k,\widetilde\eta) 
 \simeq\prod_{k\in\mathbb Z}\FF\ell_k,
\]
precomposition by \(\Phi\) sends \((a_k\ell_k)\) to \((b_k\ell_k)\), where
\begin{equation}\label{eq:dual-hecke}
 b_k=
 \begin{cases}
 a_{k+1},&k\le-1,\\
 \delta_{\vec0,\vec r}a_1,&k=0,\\
 \delta_{\vec0,\vec r}a_0,&k=1,\\
 a_{k-1},&k\ge2.
 \end{cases}
\end{equation}
\end{lemma}
\begin{proof}
For $m\ge0$, since \(1-\varpi_L\varpi_F^m[t]\in1+\varpi_L\cO_L\), we have
\(\widetilde\eta(1-\varpi_L\varpi_F^m[t])=1\).
For \(k\le0\), Lemmas~\ref{lem:local-hecke} and
\ref{lem:local-hom} give
\[
 ((\Phi_k^{k-1})^*\ell_{k-1})([1,\vec e_{\vec j}])
 =\delta_{\vec0,\vec r}\sum_{t\in k_F}t^{R(\vec j)}.
\]
If \(\vec r\ne\vec0\), $\delta_{\vec0,\vec r}=0$; if
\(\vec r=\vec0\), \(\sum_{t\in k_F}t^{R(\vec j)}=0\). Therefore $(\Phi_k^{k-1})^*\ell_{k-1}=0$. 
Also \((\Phi_k^{k+1})^*\ell_{k+1}=\ell_k\) for \(k\le-1\), whereas
\((\Phi_0^1)^*\ell_1=\delta_{\vec0,\vec r}\ell_0\).

Similarly, for \(k\ge1\),
\[
 ((\Phi_k^{k+1})^*\ell_{k+1})([1,\vec e_{\vec j}])
 =\delta_{\vec0,\vec r}
   \sum_{t\in k_F}(\bar u t)^{R(\vec r)-R(\vec j)}=0.
\]
For \(k\ge2\), \((\Phi_k^{k-1})^*\ell_{k-1}=\ell_k\), while
\((\Phi_1^0)^*\ell_0=\delta_{\vec0,\vec r}\ell_1\).
Combining these identities proves \eqref{eq:dual-hecke}.
\end{proof}
\begin{theorem}\label{thm:restriction}
Let \(\sigma\) be a Serre weight, let \(\lambda\in\FF\), and let
\(\eta:\cO_L^\times\to\FF^\times\) be a smooth character. Then
\[
 \dim_{\FF}\Hom_{\cO_L^\times}
 \left(\left(\ind^G_{KZ}\sigma/(\Phi-\lambda)\right)|_{\cO_L^\times},\eta\right)
 =
 \begin{cases}
 1,& \text{ if }\dim\sigma=1,\ \lambda^2=1,\ \eta=\sigma|_{\cO_L^\times},\\
 0,&\text{otherwise}.
 \end{cases}
\]
\end{theorem}

\begin{proof}
Since twisting by a character commutes with $\ind_{KZ}^G$ and induces an isomorphism between the Hecke algebras, it suffices to prove the theorem for
\(\sigma=\sigma_{\vec r}\). If \(\eta|_{\cO_F^\times}\ne
\zeta|_{\cO_F^\times}\), then $\Hom_{\cO_L^\times}
 (\ind^G_{KZ}\sigma/(\Phi-\lambda),\eta)=0$. So we assume \(\eta|_{\cO_F^\times}=\zeta|_{\cO_F^\times}\).
We have
\[
\Hom_{\cO_L^\times}
 (\ind^G_{KZ}\sigma/(\Phi-\lambda),\eta)=\Hom_{H}
 (\ind^G_{KZ}\sigma/(\Phi-\lambda),\widetilde\eta).
\]
Applying $\Hom_{H}
 (-,\widetilde\eta)$ to the right-exact presentation
\[
 \ind^G_{KZ}\sigma\xrightarrow{\Phi-\lambda}\ind^G_{KZ}\sigma
 \longrightarrow \ind^G_{KZ}\sigma/(\Phi-\lambda)\longrightarrow0,
\]
we obtain a left-exact sequence
\[
0\to\Hom_{H}
 (\ind^G_{KZ}\sigma/(\Phi-\lambda),\widetilde\eta)\to \Hom_{H}
 (\ind^G_{KZ}\sigma,\widetilde\eta)\xrightarrow{\Phi^*-\lambda}\Hom_{H}
 (\ind^G_{KZ}\sigma,\widetilde\eta).
\]
Thus, in the coordinates of Lemma~\ref{lem:dual-hecke}, the elements of $\ker(\Phi^*-\lambda)$ 
are precisely the sequences $(a_k)$ satisfying
\begin{equation}\label{eq:a_k}
\begin{aligned}
 a_{k+1}&=\lambda a_k\quad(k\le-1),\\
 \delta_{\vec0,\vec r}a_1&=\lambda a_0,\\
 \delta_{\vec0,\vec r}a_0&=\lambda a_1,\\
 a_{k-1}&=\lambda a_k \quad (k\ge2).
\end{aligned}
\end{equation}
When \(\lambda=0\), the first relation of \eqref{eq:a_k} kills all \(a_k\) for
\(k\le0\), and the last kills all \(a_k\) for \(k\ge1\).
When \(\lambda\ne0\) and \(\vec r\ne\vec0\), the middle
relations of \eqref{eq:a_k} give \(a_0=a_1=0\), and the first and last relations give
\(a_k=0\) for every \(k\).

Suppose that \(\vec r=\vec0\) and \(\lambda\ne0\).  Then the central character $\zeta$ of $\sigma$ is trivial. Since
\(\cO_L^\times=\cO_F^\times(1+\varpi_L\cO_L)\), the assumption \(\eta|_{\cO_F^\times}=\zeta|_{\cO_F^\times}\) implies
 \(\eta=1\). 
The middle relations of \eqref{eq:a_k} have a solution $(a_0,a_1)\neq0$
if and only if \(\lambda^2=1\). So if \(\lambda^2\neq1\), then $a_0=a_1=0$; hence the first and last relations of \eqref{eq:a_k} imply
\(a_k=0\) for every \(k\).
When \(\lambda^2=1\), \eqref{eq:a_k} is equivalent to 
\[
 a_k=\lambda^k a_0\qquad(k\in\mathbb Z),
\]
and the solution space has dimension one. 

Thus $\Hom_{\cO_L^\times}
 (\ind^G_{KZ}\sigma/(\Phi-\lambda),\eta)$ is nonzero precisely when \(\vec r=\vec0\),
\(\lambda^2=1\), and \(\eta=1\); in this case it is one-dimensional.
This proves the theorem.\end{proof}

\begin{corollary}\label{cor:unit-vanishing}
Let \(\pi\) be an infinite-dimensional irreducible smooth
representation of \(G\) with a central character, and let
\(\eta:\cO_L^\times\to\FF^\times\) be a smooth character. Then
\[
 \Hom_{\cO_L^\times}(\pi|_{\cO_L^\times},\eta)=0.
\]
\end{corollary}
\begin{proof}
Since \(\pi\) has a central character, by \cite[Theorem 33(1)]{BL94}, there exist a Serre weight $\sigma=\sigma_{\vec r}$, a smooth character $\chi \colon F^{\times}\to\FF^{\times}$, and $\lambda\in \FF$ such that $\pi$ is a quotient of 
\[
\chi\circ\det\otimes\ind^G_{KZ}\sigma/(\Phi-\lambda).
\] 
If \(\vec r=\vec0\) and \(\lambda^2=1\), every irreducible quotient of
this representation is one-dimensional by \cite[Theorem 33(3)]{BL94}.
Since \(\pi\) is infinite-dimensional, we therefore have
\(\vec r\ne\vec0\) or \(\lambda^2\ne1\).
Precomposition with the quotient map $\chi\circ\det\otimes\ind^G_{KZ}\sigma/(\Phi-\lambda)\twoheadrightarrow \pi$ gives an injection
\[
\begin{aligned}
 \Hom_{\cO_L^\times}(\pi|_{\cO_L^\times},\eta)
 &\hookrightarrow
 \Hom_{\cO_L^\times}\!\left(
 (\chi\circ\det)\otimes
 \left(\ind^G_{KZ}\sigma/(\Phi-\lambda)\right),\eta\right)\\
 &\simeq
 \Hom_{\cO_L^\times}\!\left(
 \ind^G_{KZ}\sigma/(\Phi-\lambda),
 \eta\cdot(\chi^{-1}\circ N_{L/F})\right),
\end{aligned}
\]
where \(N_{L/F}\) denotes the norm map. The last Hom space vanishes by Theorem~\ref{thm:restriction}, proving the claim.
\end{proof}

\begin{corollary}\label{cor:torus-vanishing}
Let \(\pi\) be an infinite-dimensional irreducible smooth
representation of \(G\) with a central character, and let \(\eta:L^\times\to\FF^\times\)
be a smooth character. Then
\[
 \Hom_{L^\times}(\pi|_{L^\times},\eta)=0.
\]
\end{corollary}
\begin{proof}
We have an injection
\[
 \Hom_{L^\times}(\pi|_{L^\times},\eta)
 \hookrightarrow
 \Hom_{\cO_L^\times}(\pi|_{\cO_L^\times},\eta|_{\cO_L^\times}),
\]
where the target vanishes by Corollary~\ref{cor:unit-vanishing}.
\end{proof}

\begin{corollary}\label{cor:open-subgroup}
Let \(\pi\) be an infinite-dimensional irreducible smooth
representation of \(G\) with a central character. Let \(P\) be an open pro-\(p\) subgroup
of \(\cO_L^\times\). Then
\[
 \Hom_{P}(\pi|_{P},\FF)=0.
\]
\end{corollary}
\begin{proof}
Frobenius reciprocity gives
\begin{equation}\label{eq:Ind}
 \Hom_{P}(\pi|_{P},\FF)
 \simeq
 \Hom_{\cO_L^\times}
 \left(\pi|_{\cO_L^\times},\Ind_{P}^{\cO_L^\times}\FF\right).
\end{equation}
The induced representation $\Ind_{P}^{\cO_L^\times}\FF$ is finite-dimensional because \(P\) has
finite index in the compact group \(\cO_L^\times\).
As \(\cO_L^\times\) is abelian and \(\FF\) is algebraically closed, the
finite-dimensional representation \(\Ind_{P}^{\cO_L^\times}\FF\)
has a finite filtration with smooth characters as successive quotients.
Corollary~\ref{cor:unit-vanishing} and left exactness of
\(\Hom_{\cO_L^\times}(\pi,-)\) show, by induction on the filtration,
that the right-hand side of \eqref{eq:Ind} is zero.
\end{proof}

\section{Application to multilinear forms}
\label{sec:multilinear}
Let \(P\) be a pro-$p$ group. Let $I$ be the augmentation ideal of $\FF[\![P]\!]$. For an $\FF[\![P]\!]$-module $V$, write $V[I^n] = \{v\in V~|~ I^n v = 0\}$.
\begin{lemma}\label{smooth}
Suppose $V$ is a smooth $\FF$-representation of $P$. Then $V$ is an $\FF[\![P]\!]$-module and $V=\bigcup_{n\ge1}V[I^n]$.
\end{lemma}
\begin{proof}
Since $V$ is the union of finite-dimensional $P$-subrepresentations, it is enough to assume $V$ is finite-dimensional. Then there exists an open normal subgroup $P'$ of $P$ such that $P'$ acts trivially on $V$. The Iwasawa algebra $\FF[\![P/P']\!]\cong\FF[P/P']$ acts on $V$.  Since $P/P'$ is a finite $p$-group, $\FF[P/P']$ is local Artinian, and its maximal ideal $I'$ is nilpotent. Thus $(I')^nV=0$ for some integer $n\ge1$, and hence $I^nV=0$.
\end{proof}
For two smooth $\FF$-representations $V_1$ and $V_2$ of $P$, we write $V_1\otimes_{\FF} V_2$ for the representation of $P$ with diagonal action.
\begin{lemma}\label{lem:tensor-vanishing}
Let \(V_1\) and \(V_2\) be smooth
\(\mathbb F\)-representations of a pro-$p$ group \(P\). If
\[
  \FF\otimes_{\FF[\![P]\!]}V_1=0,
\]
then
\[
\FF\otimes_{\FF[\![P]\!]}(V_1\otimes_{\mathbb F}V_2)=0.
\]
In particular,
\[
  \Hom_P(V_1\otimes_{\mathbb F}V_2,\mathbb F)=0.
\]
Here \(\mathbb F\) carries the trivial action.
\end{lemma}

\begin{proof}
By Lemma~\ref{smooth}, we have $V_2=\bigcup_{n\ge1}V_2[I^n]$. Since tensor products commute with direct limits, we only need to show $\FF\otimes_{\FF[\![P]\!]}(V_1\otimes_{\mathbb F}V_2[I^n])=0$ for each $n\ge1$. The quotients $V_2[I^{n+1}]/V_2[I^n]$ have trivial $P$-action, since $I V_2[I^{n+1}]\subseteq V_2[I^n]$. By right exactness and dévissage, we only need to deal with the case when $V_2=V_2[I]$. In this case $P$ acts trivially on $V_2$, so
\[
\FF\otimes_{\FF[\![P]\!]}(V_1\otimes_{\mathbb F}V_2)\cong (\FF\otimes_{\FF[\![P]\!]}V_1)\otimes_{\mathbb F}V_2=0,
\]
as required.
\end{proof}

\begin{corollary}\label{cor:multilinear}
Let \(n\geq 1\), and let \(\pi_1,\ldots,\pi_n\) be smooth \(\mathbb F\)-representations of \(G=\GL_2(F)\). Suppose that at least one \(\pi_i\) is
infinite-dimensional and irreducible, with a central character. Then
\[
  \Hom_G\left(
  \pi_1\otimes_{\mathbb F}\cdots\otimes_{\mathbb F}\pi_n,
  \mathbb F\right)=0.
\]
\end{corollary}

\begin{proof}
After reordering the factors, we may assume that \(\pi_1\) is infinite-dimensional and irreducible, and has a central character. Set \(P=1+\varpi_L\cO_L\). Corollary~\ref{cor:open-subgroup} gives
\[
  \Hom_{P}(\pi_1|_{P},\mathbb F)=0.
\]
So $\FF\otimes_{\FF[\![P]\!]}\pi_1=0$.
Applying Lemma~\ref{lem:tensor-vanishing} to
\(\pi_1\) and \(\pi_2\otimes_{\mathbb F}\cdots\otimes_{\mathbb F}\pi_n\) with the latter tensor product interpreted as \(\FF\) when \(n=1\), we obtain
\[
  \Hom_{P}\left(
  \pi_1\otimes_{\mathbb F}\cdots\otimes_{\mathbb F}\pi_n,
  \mathbb F\right)=0.
\]
Every \(G\)-equivariant linear functional is
\(P\)-equivariant, so the claimed vanishing follows.
\end{proof}

\begin{remark}\label{rem:dospinescu}
After receiving an earlier version of this paper, Gabriel Dospinescu communicated to us the following direct proof of
Corollary~\ref{cor:multilinear}, using Corollary~\ref{cor:open-subgroup}. After reordering the factors, we may assume
that \(\pi_1\) is infinite-dimensional and irreducible, with a central
character. The case \(n=1\) follows directly from
Corollary~\ref{cor:open-subgroup}. Suppose \(n\geq2\) and that a nonzero
\(G\)-invariant linear form \(\ell\) exists. Then there exists a pure tensor
\(v_1\otimes_{\mathbb F}\cdots\otimes_{\mathbb F}v_n\) such that \(\ell(v_1\otimes_{\mathbb F}\cdots\otimes_{\mathbb F}v_n)\neq0\). By smoothness, there is an open pro-\(p\) subgroup
\(P\subseteq\cO_L^\times\) fixing \(v_2,\ldots,v_n\). The map
$$
 \begin{aligned}\pi_1&\longrightarrow\mathbb F\\
 v&\longmapsto
 \ell(v\otimes_{\mathbb F}v_2\otimes_{\mathbb F}\cdots
 \otimes_{\mathbb F}v_n)
 \end{aligned}
$$
is then a nonzero \(P\)-equivariant linear form, contradicting
Corollary~\ref{cor:open-subgroup}.
\end{remark}

\end{document}